\documentclass[12pt]{article}
\usepackage[margin=1in]{geometry}
\usepackage{amsthm, amsmath,amsfonts,amssymb,mathrsfs}
\usepackage{hyperref, url}
\usepackage{stmaryrd}
\makeatletter
\newcommand{\subjclass}[2][2020]{%
  \let\@oldtitle\@title%
  \gdef\@title{\@oldtitle\footnotetext{#1 \emph{Mathematics subject classification}: #2}}%
}
\newcommand{\keywords}[1]{%
  \let\@@oldtitle\@title%
  \gdef\@title{\@@oldtitle\footnotetext{\emph{Key words and phrases}: #1}}%
}
\makeatother
\newtheorem{theorem}{Theorem}[section]
\newtheorem*{theorem*}{Theorem}
\newtheorem{lemma}[theorem]{Lemma}
\newtheorem{proposition}[theorem]{Proposition}

\theoremstyle{remark}

\newtheorem{remark}[theorem]{\bf Remark}

\newtheorem{conjecture}[theorem]{\bf Conjecture}
\newtheorem*{acknowledgements}{\bf Acknowledgements}
\numberwithin{equation}{section}

\begin{document}
\title{On the $P(t)$-adic Littlewood conjecture in odd characteristics}

\author{Li Lai, Johannes Sprang}
\date{}

\maketitle

\begin{abstract}
The $P(t)$-adic Littlewood conjecture is a function field analogue of the famous $p$-adic Littlewood conjecture in Diophantine approximation.
In this paper, we prove that the $P(t)$-adic Littlewood conjecture fails for any irreducible polynomial $P(t)$ over any ground field of odd characteristic.
\end{abstract}

\section{Introduction}
Littlewood conjecture is a famous unsolved problem in simultaneous Diophantine approximation, which states that
\[
\inf_{q \in \mathbb{Z}\setminus\{0\}} |q|\cdot\lVert q\alpha \rVert \cdot \lVert q\beta \rVert = 0
\]
for any two real numbers $\alpha$ and $\beta$.
Here, $\lVert \cdot \rVert$ denotes the distance of a real number to the nearest integer(s). 
We refer the reader to Bugeaud's survey \cite{Bug2014} for the history and recent developments. 

The $p$-adic Littlewood conjecture, formulated by de Mathan and Teuli\'e \cite{dMT2004} in 2004, is an analogue of the classical Littlewood conjecture.
It states that for any real number $\alpha$ and any prime $p$, we have
\[
\inf_{q \in \mathbb{Z}\setminus\{0\}} |q|\cdot |q|_p\cdot \lVert q\alpha\rVert=0,
\]
where $|q|_p=p^{-v_p(q)}$ denotes the $p$-adic absolute value of an integer $q$, normalized by $v_p(p)=1$.
It is believed that the $p$-adic Littlewood conjecture is simpler than the classical one, but still deep.

The analogue conjecture over function fields has also been formulated by de Mathan and Teuli\'e in 2004:

\begin{conjecture}[the $P(t)$-adic Littlewood conjecture \cite{dMT2004}]\label{conj:P(t)-LC}
Let $\mathbb{K}$ be an arbitrary field and $P(t) \in \mathbb{K}[t]$ an irreducible polynomial over $\mathbb{K}$.
Then, for any $\Gamma(t) \in \mathbb{K}(\!(t^{-1})\!)$, we have
\[
\inf_{Q(t) \in \mathbb{K}[t]\setminus\{0\}} \left| Q(t) \right| \cdot \left| Q(t) \right|_{P(t)} \cdot  \left| \left\langle Q(t) \Gamma(t) \right\rangle \right| = 0.
\]
\end{conjecture}
Here, for any nonzero formal Laurent series $\Gamma(t)=\sum_{k\geqslant -k_0}a_kt^{-k}\in \mathbb{K}(\!(t^{-1})\!)$ with $a_{-k_0} \neq 0$, we define
\[
\deg \Gamma(t) := k_0 \in \mathbb{Z} \quad \text{and}\quad |\Gamma(t)|:=2^{\deg \Gamma(t)},
\]
and let
\[
\langle \Gamma(t) \rangle:=\sum_{k\geqslant 1} a_kt^{-k} \in \mathbb{K}(\!(t^{-1})\!)
\]
denote the fractional part of $\Gamma(t)$. By convention, $\deg 0 = -\infty$ and $\langle  0 \rangle = 0$.
Note that $\mathbb{K}(\!(t^{-1})\!)$ is the completion of $\mathbb{K}(t)$ with respect to the absolute value $|\cdot|$. 
For each $Q(t) \in \mathbb{K}[t]$, the $P(t)$-adic absolute value is given by $|Q(t)|_{P(t)}:=|P(t)|^{-v_{P(t)}(Q(t))}$, with the normalization $v_{P(t)}(P(t))=1$.

Conjecture \ref{conj:P(t)-LC} has been extensively studied in recent years, and many progresses have been made. Already, de Mathan and Teuli\'e \cite{dMT2004} observed that there are counterexamples to the $P(t)$-adic Littlewood Conjecture when $\mathbb{K}$ is an infinite field. 
By an elementary and important observation of Roberson \cite[Lemma 2.0.1]{Rob2026}, any counterexample to Conjecture \ref{conj:P(t)-LC} for the case $P(t)=t$ induces counterexamples to Conjecture \ref{conj:P(t)-LC} for all irreducible polynomials $P(t)$. (See Lemma \ref{lem:t_to_P(t)}.)
In 2021, Adiceam, Nesharim and Lunnon \cite{ANL2021} constructed an explicit counterexample to Conjecture \ref{conj:P(t)-LC} in the case $\operatorname{char} \mathbb{K}=3$. 
Recently, Garrett and Robertson \cite{GR2026} disproved Conjecture \ref{conj:P(t)-LC} in the case $\operatorname{char} \mathbb{K} \in \{5,7,11\}$. 
A significant progress was made by Adiceam and Badziahin \cite{AB2025+}; they disproved Conjeture \ref{conj:P(t)-LC} in the case $\operatorname{char} \mathbb{K} \equiv 3 \pmod{4}$. 

In this paper, we disprove Conjecture \ref{conj:P(t)-LC} over any field of odd characteristic, by explicitly constructing a counterexample. 
Our main result is the following.

\begin{theorem}[The $P(t)$-adic Littlewood conjecture fails over fields of odd characteristics]\label{thm:main}
Let $\mathbb{K}$ be an arbitrary field of odd characteristic $p$, and $\mathbb{F}_p$ the prime field contained in $\mathbb{K}$.
Let $P(t) \in \mathbb{K}[t]$ be any irreducible polynomial over $\mathbb{K}$.
Denote by 
\[
r= v_2(p-1) \in \mathbb{Z}_{\geqslant 1}
\] 
and 
\[
\mu^{*}_{2^r} = \left\{  \zeta \in \mathbb{F}_p   \mid  \zeta^{2^{r-1}} = -1 \right\}.
\]
Consider
\begin{equation}\label{eqn:construction}
\Lambda(t) = \sum_{k=0}^{\infty} \sum_{\zeta \in \mu^{*}_{2^r}} \frac{t^{2^k}}{t^{2^{k+1}}-\zeta} \in \mathbb{F}_p(\!(t^{-1})\!).
\end{equation}
Then, we have
\[
\inf_{Q(t) \in \mathbb{K}[t]\setminus\{0\}} \left| Q(t) \right| \cdot \left| Q(t) \right|_{P(t)} \cdot  \left|  \left\langle Q(t) \cdot \Lambda(P(t)) \right\rangle \right| \geqslant 2^{-2^{r+1}\deg P(t)}.
\]
Therefore, the formal Laurent series $\Lambda(P(t)) \in \mathbb{K}(\!(t^{-1})\!)$ is a counterexample to the $P(t)$-adic Littlewood conjecture over $\mathbb{K}$.
\end{theorem}

\begin{remark}
Since 
\[
\left|  \sum_{\zeta \in \mu^{*}_{2^r}} \frac{t^{2^k}}{t^{2^{k+1}}-\zeta} \right| \leqslant 2^{-2^k},
\]
the right-hand side of \eqref{eqn:construction} indeed converges in $\mathbb{F}_p(\!(t^{-1})\!)$. 
Moreover, our construction $\Lambda(t)$ is characterized by the following ``functional equation'':
\begin{equation}\label{eqn:functional_equation}
\Lambda(t) = \Lambda(t^2) + \sum_{\zeta \in \mu^{*}_{2^r}} \frac{t}{t^{2}-\zeta}. 
\end{equation}
\end{remark}

We point out that, Theorem \ref{thm:main} not only completely disproves Conjecture \ref{conj:P(t)-LC} in odd characteristics, but also our proof is far simpler than those presented in \cite{AB2025+, ANL2021, GR2026}.
For $r=v_2(p-1) \geqslant 2$, our counterexample $\Lambda(t)$ is distinct from, yet analogous to, the conjectural counterexample proposed by Garret and Robertson in \cite{GR2026}. 
When $r=1$, our counterexample $\Lambda(t)$ coincides with Adiceam and Badziahin's counterexample in \cite{AB2025+}.

The structure of this paper is as follows. 
In Section \ref{sec:2}, we reduce Theorem \ref{thm:main} to a key property of $\Lambda(t)$; namely, Proposition \ref{prop:key}.
In Section \ref{sec:3}, we prove Proposition \ref{prop:key}.

\section{Reduction to a key property of $\Lambda(t)$}\label{sec:2}

We first state an important observation of Robertson.

\begin{lemma}[Robertson {\cite[Lemma 2.0.1]{Rob2026}}]\label{lem:t_to_P(t)}
Let $\mathbb{K}$ be any field. 
Assume that $\Gamma(t) \in \mathbb{K}(\!(t^{-1})\!)$ fails the $t$-adic Littlewood conjecture in the sense that 
\[
\inf_{Q(t) \in \mathbb{K}[t]\setminus\{0\}} \left| Q(t) \right| \cdot \left| Q(t) \right|_{t} \cdot  \left|  \left\langle Q(t) \cdot \Gamma(t) \right\rangle \right| = 2^{-\rho}
\] 
for some integer $\rho$.
Then, the Laurent series $\Gamma(P(t)) \in \mathbb{K}(\!(t^{-1})\!)$ fails the $P(t)$-adic Littlewood conjecture in the sense that
\[
\inf_{Q(t) \in \mathbb{K}[t]\setminus\{0\}} \left| Q(t) \right| \cdot \left| Q(t) \right|_{P(t)} \cdot  \left| \left\langle Q(t) \cdot \Gamma(P(t)) \right\rangle \right| = 2^{-\rho \cdot \deg P(t)}.
\]
\end{lemma}

The next lemma is a standard result from elementary linear algebra. It essentially appeared in \cite[Theorem 2.2]{AB2025+}. We include its short proof here for the reader’s convenience.

\begin{lemma}\label{lem:KtoFp}
Let $\mathbb{K}$ be any field and $\mathbb{F}$ a subfield of $\mathbb{K}$.
Let $Q(t)\in \mathbb{K}[t]\setminus \{0\}$ and $\Gamma(t) \in \mathbb{F}(\!(t^{-1})\!)$.
Then, there exists $Q^{*}(t) \in \mathbb{F}[t]\setminus\{0\}$ such that
\[
\deg Q^*(t) \leqslant \deg Q(t) \quad\text{and}\quad \deg \langle Q^*(t)\Gamma(t) \rangle \leqslant \deg \langle Q(t)\Gamma(t) \rangle.
\]
\end{lemma}

\begin{proof}
We may assume $\Gamma(t) \neq 0$.
Suppose $Q(t)=\sum_{k=0}^{N}a_kt^{k}$ with $a_N \neq 0$ and $\Gamma(t)=\sum_{k \geqslant -k_0} b_kt^{-k}$. 
Note that for any $k \in \mathbb{Z}_{\geqslant 1}$, the coefficient of $t^{-k}$ in $\langle Q(t)\Gamma(t)\rangle$ is $\sum_{j=0}^{N} b_{k+j}a_j$. 
Therefore, if $\deg \langle Q(t)\Gamma(t)\rangle = -M-1$, then 
\[
\boldsymbol{B}\boldsymbol{a}=\boldsymbol{0}, \quad\text{where } \boldsymbol{B}=(b_{i+j-1}) \in \operatorname{Mat}_{M \times (N+1)}(\mathbb{F}) \text{ and } \boldsymbol{a}= (a_{i-1}) \in \operatorname{Mat}_{(N+1)\times 1}(\mathbb{K}).
\]
(If $M=0$, we can take $Q^{*}(t)=1$ to finish the proof. 
In the following, we assume $M \in \mathbb{Z}_{\geqslant 1}$.)
Since $\boldsymbol{a} \neq \boldsymbol{0}$, we deduce that $\boldsymbol{B} \in \operatorname{Mat}_{M \times (N+1)}(\mathbb{F})$ does not have full column-rank. 
Therefore, there exists $\boldsymbol{a}^{*}=(a_{i-1}^*) \in \operatorname{Mat}_{(N+1)\times 1}(\mathbb{F}) \setminus \{\boldsymbol{0}\}$ such that $\boldsymbol{B}\boldsymbol{a}^{*}=\boldsymbol{0}$. 
Thus, the polynomial $Q^*(t)=\sum_{k=0}^{N} a_{k}^* t^k$ satisfies all requirements.
\end{proof}

The following Proposition \ref{prop:key} is the key property of $\Lambda(t)$ constructed in \eqref{eqn:construction}. 

\begin{proposition}[The key property of $\Lambda(t)$]\label{prop:key}
Let $p$ be any odd prime and $r=v_2(p-1)$.
Let $\Lambda(t) \in \mathbb{F}_p(\!(t^{-1})\!)$ be the Laurent series defined in \eqref{eqn:construction}.
Then, there does not exist a polynomial $Q(t) \in \mathbb{F}_p[t] \setminus \{0\}$ and a nonnegative integer $m$ such that
\[
\left| \left\langle Q(t) \cdot t^m\Lambda(t)  \right\rangle \right| < \frac{1}{|Q(t)|} \quad\text{and}\quad t^{2^{r}}-1 \mid Q(t).
\] 
\end{proposition}

Proposition \ref{prop:key} will be proved in the next section. Now, we prove Theorem \ref{thm:main} assuming Proposition \ref{prop:key}.

\begin{proof}[Proof of Theorem \textup{\ref{thm:main}} assuming Proposition \textup{\ref{prop:key}}]
By Lemma \ref{lem:t_to_P(t)}, it suffices to prove that 
\[
\inf_{Q(t) \in \mathbb{K}[t]\setminus\{0\}} \left| Q(t) \right| \cdot \left| Q(t) \right|_{t} \cdot  \left|  \left\langle Q(t) \cdot \Lambda(t) \right\rangle \right| \geqslant 2^{-2^{r+1}}.
\]
Since every polynomial $Q(t) \in \mathbb{K}[t]\setminus\{0\}$ can be written as $Q(t)=t^m \widetilde{Q}(t)$ for some nonnegative integer $m$ and some polynomial $\widetilde{Q}(t) \in \mathbb{K}[t]$ with $\widetilde{Q}(0) \neq 0$, it sufficies to prove that
\[
\inf_{\substack{Q(t) \in \mathbb{K}[t]\setminus\{0\} \\ m \in \mathbb{Z}_{\geqslant 0}}} |Q(t)| \cdot \left|  \left\langle Q(t) \cdot t^m\Lambda(t) \right\rangle \right| \geqslant 2^{-2^{r+1}}.
\]

We argue by contradiction. 
Suppose that there exist a polynomial $Q(t) \in \mathbb{K}[t]\setminus\{0\}$ and a nonnegative integer $m$ such that $|Q(t)| \cdot \left|  \left\langle Q(t) \cdot t^m\Lambda(t) \right\rangle \right| < 2^{-2^{r+1}}$. 
Note that $t^m\Lambda(t) \in \mathbb{F}_p(\!(t^{-1})\!)$.
By Lemma \ref{lem:KtoFp}, there exists a $Q^{*}(t) \in \mathbb{F}_p[t]\setminus \{ 0 \}$ such that $|Q^{*}(t)| \leqslant |Q(t)|$ and $\left|  \left\langle Q^{*}(t) \cdot t^m\Lambda(t) \right\rangle \right| \leqslant \left|  \left\langle Q(t) \cdot t^m\Lambda(t) \right\rangle \right|$.
Take $\widehat{Q}(t) = (t^{2^r}-1)Q^{*}(t)$, then 
\begin{align*}
\left|\widehat{Q}(t)\right| \cdot \left| \left\langle \widehat{Q}(t) \cdot t^m\Lambda(t) \right\rangle \right| &\leqslant 2^{2^r} |Q^*(t)| \cdot 2^{2^r} \left|  \left\langle Q^*(t) \cdot t^m\Lambda(t) \right\rangle \right| \\
&\leqslant 2^{2^{r+1}} |Q(t)| \cdot \left|  \left\langle Q(t) \cdot t^m\Lambda(t) \right\rangle \right| < 1.
\end{align*}
Thus, we have
\[ 
\widehat{Q}(t) \in \mathbb{F}_p[t] \setminus \{0\}, \quad \left| \left\langle \widehat{Q}(t) \cdot t^m\Lambda(t)  \right\rangle \right| < \frac{1}{\left|\widehat{Q}(t)\right|} \quad\text{and}\quad t^{2^{r}}-1 \mid \widehat{Q}(t),
\]
which contradicts Proposition \ref{prop:key}.
\end{proof}

\section{Proof of the key property of $\Lambda(t)$}\label{sec:3}

In this section, we prove Proposition \ref{prop:key}.
The following big $O$ notation is convenient: For $\Gamma(t) \in \mathbb{K}(\!(t^{-1})\!)$ and $n \in \mathbb{Z}$, we say $\Gamma(t)=O(t^{n})$ if $\deg \Gamma(t) \leqslant n$. 
Recall the convention $\deg 0 = -\infty$.

\begin{proof}[Proof of Proposition \textup{\ref{prop:key}}]
We proceed by contradiction.
Suppose there exists a pair $(Q(t),m)$, where $Q(t) \in \mathbb{F}_p[t] \setminus \{0\}$ and $m \in \mathbb{Z}_{\geqslant 0}$, such that
\begin{equation}\label{eqn:impossible(Q,m)}
\left| \left\langle Q(t) \cdot t^m\Lambda(t)  \right\rangle \right| < \frac{1}{|Q(t)|} \quad\text{and}\quad t^{2^{r}}-1 \mid Q(t).
\end{equation}
We may then choose such a pair $(Q(t),m)$ with $\deg Q(t)$ minimal.
We shall construct another pair $(\widehat{Q}(t),\widehat{m})$ satisfying \eqref{eqn:impossible(Q,m)} with $\deg \widehat{Q}(t) < \deg Q(t)$. 

Since $t^{2^r}-1 \mid Q(t)$ and $Q(t) \neq 0$, we have
\begin{equation}\label{eqn:degQ>}
\deg Q(t) \geqslant 2^{r}.
\end{equation}
Write $Q(t)=\left(t^{2^r}-1\right)\widetilde{Q}(t)$ and $\widetilde{Q}(t)=\widetilde{Q}_0(t^2)+t\widetilde{Q}_1(t^2)$, where $\widetilde{Q}_0(t), \widetilde{Q}_1(t) \in \mathbb{F}_p[t]$.
Let
\begin{equation}\label{eqn:tildeQ0Q1}
Q_0(t) = \left( t^{2^{r-1}} - 1\right) \cdot \widetilde{Q}_0(t) \quad\text{and}\quad Q_1(t) = \left( t^{2^{r-1}} - 1\right) \cdot \widetilde{Q}_1(t).
\end{equation}
Then, $Q(t)=Q_0(t^2)+tQ_1(t^2)$ and $Q_0(t),Q_1(t) \in \mathbb{F}_p[t]$. Suppose that 
\[
\text{$\deg Q(t) = 2q+\epsilon$ and $m=2n+\delta$,\ where $q,n \in \mathbb{Z}_{\geqslant 0}$ and $\epsilon,\delta\in \{0,1\}$.}
\]
Then, 
\begin{center}
\begin{tabular}{ll}
$\deg Q_\epsilon(t) = q$,  & \quad  $\deg \widetilde{Q}_\epsilon(t) = q-2^{r-1}$, \\
$\deg Q_{1-\epsilon}(t) \leqslant q-1+\epsilon$,  & \quad $\deg \widetilde{Q}_{1-\epsilon}(t) \leqslant q-1+\epsilon-2^{r-1}$.
\end{tabular}
\end{center} 
By \eqref{eqn:degQ>}, we have $q \geqslant 2^{r-1}$. 
In particular, $Q_\epsilon(t) \in \mathbb{F}_p[t] \setminus \{0\}$ and $\deg Q_\epsilon(t) < \deg Q(t)$.

By \eqref{eqn:functional_equation}, we have
\[
\left\langle t^{2n+\delta} \Lambda(t) \right\rangle = \left\langle t^{2n+\delta}\Lambda(t^2) \right\rangle + \sum_{\zeta \in \mu^{*}_{2^r}} \frac{\zeta^{n+\delta} t^{1-\delta}}{t^{2}-\zeta}.
\]
Therefore,
\begin{align*}
&\ \left\langle Q(t) \cdot t^m\Lambda(t)  \right\rangle = \left\langle \left( Q_0(t^2)+tQ_1(t^2) \right) \cdot \left( t^{2n+\delta} \Lambda(t^2) + \sum_{\zeta \in \mu^{*}_{2^r}} \frac{\zeta^{n+\delta} t^{1-\delta}}{t^{2}-\zeta} \right)  \right\rangle \\
=&\ \underbrace{\left\langle Q_\delta(t^2) \cdot t^{2n+2\delta}\Lambda(t^2) +   \sum_{\zeta \in \mu^{*}_{2^r}} \frac{\zeta^{n+1}Q_{1-\delta}(\zeta)}{t^2-\zeta} \right\rangle}_{\text{even}} + \underbrace{ t \cdot \left\langle Q_{1-\delta}(t^2)\cdot t^{2n}\Lambda(t^2) + \sum_{\zeta \in \mu^{*}_{2^r}} \frac{\zeta^{n+\delta} Q_\delta(\zeta)}{t^2-\zeta} \right\rangle}_{\text{odd}}.
\end{align*}
By hypothesis, we have $\left\langle Q(t) \cdot t^m\Lambda(t)  \right\rangle = O(t^{-2q-1-\epsilon})$.
Thus,
\begin{align*}
\left\langle Q_\delta(t^2) \cdot t^{2n+2\delta}\Lambda(t^2) +   \sum_{\zeta \in \mu^{*}_{2^r}} \frac{\zeta^{n+1}Q_{1-\delta}(\zeta)}{t^2-\zeta} \right\rangle &= O(t^{-2q-2}), \\
\left\langle Q_{1-\delta}(t^2)\cdot t^{2n}\Lambda(t^2) + \sum_{\zeta \in \mu^{*}_{2^r}} \frac{\zeta^{n+\delta} Q_\delta(\zeta)}{t^2-\zeta} \right\rangle &= O(t^{-2q-2-2\epsilon}). 
\end{align*}
In other words,
\begin{align}
\left\langle Q_\delta(t) \cdot t^{n+\delta}\Lambda(t) + \sum_{\zeta \in \mu^{*}_{2^r}} \frac{\zeta^{n+1}Q_{1-\delta}(\zeta)}{t-\zeta} \right\rangle &= O(t^{-q-1}), \label{1.1}\\
\left\langle Q_{1-\delta}(t)\cdot t^{n}\Lambda(t) + \sum_{\zeta \in \mu^{*}_{2^r}} \frac{\zeta^{n+\delta} Q_\delta(\zeta)}{t-\zeta} \right\rangle &= O(t^{-q-1-\epsilon}). \label{1.2}
\end{align}
By \eqref{eqn:tildeQ0Q1}, we have $Q_0(\zeta)=-2\widetilde{Q}_0(\zeta)$ and $Q_1(\zeta)=-2\widetilde{Q}_1(\zeta)$ for any $\zeta \in \mu^{*}_{2^r}$. 
Multiplying \eqref{1.1} (respectively, \eqref{1.2}) by 
\[
\widetilde{Q}_{1-\delta}(t)=O\left(t^{q-2^{r-1}}\right)\qquad (\text{respectively, } t^\delta \widetilde{Q}_{\delta}(t)=
O\left(t^{q-2^{r-1}+\varepsilon\delta}\right) )
\]
we obtain
\begin{align}
\left\langle \left(t^{2^{r-1}}-1\right)\widetilde{Q}_0(t)\widetilde{Q}_1(t) \cdot t^{n+\delta}\Lambda(t) -2 \sum_{\zeta \in \mu^{*}_{2^r}} \frac{\zeta^{n+1}\widetilde{Q}_{1-\delta}(\zeta)^2}{t-\zeta} \right\rangle &= O(t^{-2^{r-1}-1}), \label{1.3}\\
\left\langle \left(t^{2^{r-1}}-1\right)\widetilde{Q}_0(t)\widetilde{Q}_1(t) \cdot t^{n+\delta}\Lambda(t) -2 \sum_{\zeta \in \mu^{*}_{2^r}} \frac{\zeta^{n+2\delta}\widetilde{Q}_{\delta}(\zeta)^2}{t-\zeta} \right\rangle &= O(t^{-2^{r-1}-1}). \label{1.4}
\end{align}
Subtracting \eqref{1.3} by \eqref{1.4}, we deduce that
\begin{equation}\label{1.5}
\sum_{\zeta \in \mu^{*}_{2^r}} \frac{\zeta^{n+\delta}\left( \widetilde{Q}_0(\zeta)^2 - \zeta \widetilde{Q}_1(\zeta)^2 \right)}{t-\zeta} = O(t^{-2^{r-1}-1}).
\end{equation}

Noting that
\begin{align*}
&\ \sum_{\zeta \in \mu^{*}_{2^r}} \frac{\zeta^{n+\delta}\left( \widetilde{Q}_0(\zeta)^2 - \zeta \widetilde{Q}_1(\zeta)^2 \right)}{t-\zeta} \\
=&\ \sum_{k=1}^{2^{r-1}} \left( \sum_{\zeta \in \mu^{*}_{2^r}} \zeta^{n+\delta+k-1} \left( \widetilde{Q}_0(\zeta)^2 - \zeta \widetilde{Q}_1(\zeta)^2 \right) \right) t^{-k} + O\left( t^{-2^{r-1}-1} \right),
\end{align*}
Equation \eqref{1.5} implies that
\begin{equation}\label{1.6}
\sum_{\zeta \in \mu^{*}_{2^r}} \zeta^{n+
\delta+k-1} \left( \widetilde{Q}_0(\zeta)^2 - \zeta \widetilde{Q}_1(\zeta)^2 \right) = 0 \quad\text{for any}\ k=1,2,\ldots,2^{r-1}.
\end{equation}
Since the Vandermonde matrix $\left(\zeta^{k-1} \right)_{\zeta \in \mu^{*}_{2^r}, 1 \leqslant k \leqslant 2^{r-1}}$ is invertible, Equation \eqref{1.6} implies that
\[
\widetilde{Q}_0(\zeta)^2 = \zeta \widetilde{Q}_1(\zeta)^2 \quad\text{for any}\ \zeta \in \mu^{*}_{2^r}.
\]
Since $\zeta$ is not a square in $\mathbb{F}_p$ for any $\zeta \in \mu^{*}_{2^r}$, we obtain
\begin{equation}\label{1.7}
\widetilde{Q}_0(\zeta) = \widetilde{Q}_1(\zeta) = 0 \quad\text{for any}\ \zeta \in \mu^{*}_{2^r}.
\end{equation}
Now, by \eqref{eqn:tildeQ0Q1}, \eqref{1.1}, \eqref{1.2} and \eqref{1.7}, we have 
\[
\left| \left\langle Q_\epsilon(t) \cdot t^{n+\epsilon\delta}\Lambda(t)  \right\rangle \right| < \frac{1}{|Q_\epsilon(t)|} \quad\text{and}\quad t^{2^{r}}-1 \mid Q_\epsilon(t).
\]
This is a contradiction since $Q_\epsilon(t) \in \mathbb{F}_p[t]\setminus\{0\}$ and $\deg Q_\epsilon(t) < \deg Q(t)$.
\end{proof}

\begin{acknowledgements}
L.L. is supported by Research Foundation for Scholars of Xiamen University X2450218.
\end{acknowledgements}

\vspace*{3mm}
\begin{flushright}
\begin{minipage}{148mm}\sc\footnotesize
L.\,L.: School of Mathematical Sciences, Xiamen University, Fujian, China \\
{\it E-mail addresses}: \href{mailto:lilaimath@gmail.com}{{\tt lilaimath@gmail.com}}, \href{mailto:lilai@xmu.edu.cn}{{\tt lilai@xmu.edu.cn}} \vspace*{3mm}
\end{minipage}
\end{flushright}

\begin{flushright}
\begin{minipage}{148mm}\sc\footnotesize
J.\,S.: Department of Mathematics, University of Duisburg-Essen, Essen, Germany \\
{\it E-mail address}: \href{mailto:johannes.sprang@uni-due.de}{{\tt johannes.sprang@uni-due.de}} \vspace*{3mm}
\end{minipage}
\end{flushright}

\end{document}